\documentclass{article}

\usepackage{amsfonts}
\usepackage{amsmath}
\usepackage{amssymb}
\usepackage{wasysym}
\usepackage{yfonts}  
\usepackage{amsthm}
\usepackage{amscd}   
\usepackage[all]{xypic}
\usepackage{mathrsfs}
  
\CompileMatrices

\newtheorem{theorem}{Theorem}[section]
\newtheorem{lemma}[theorem]{Lemma}
\newtheorem{proposition}[theorem]{Proposition}
\newtheorem{corollary}[theorem]{Corollary} 
\newtheorem*{definition}{Definition}
\newtheorem*{example}{Example}
\newtheorem*{remark}{Remark}
\newtheorem*{notation}{Notation and terminology}
\newtheorem*{acknowledgement}{Acknowledgement}

\numberwithin{equation}{section}

\begin{document}
\title{Cohomology and profinite topologies for solvable groups of finite rank}
\author{
{\sc Karl Lorensen}\\
Pennsylvania State University, Altoona College\\
Altoona, PA 16601\\
USA\\
e-mail: {\tt kql3@psu.edu}
}

\maketitle
 
 \centerline {\it Dedicated to the memory of Peter Hilton.}
\vspace{10pt}

\begin{abstract} Assume $G$ is a solvable group whose elementary abelian sections are all finite. Suppose, further, that $p$ is a prime such that $G$ fails to contain any subgroups isomorphic to $C_{p^\infty}$. We show that if $G$ is nilpotent, then the pro-$p$ completion map $G\to \hat{G}_p$ induces an isomorphism $H^\ast(\hat{G}_p,M)\to H^\ast(G,M)$ for any discrete $\hat{G}_p$-module $M$ of finite $p$-power order.  For the general case, we prove that $G$ contains a normal subgroup $N$ of finite index such that the map $H^\ast(\hat{N}_p,M)\to H^\ast(N,M)$ is an isomorphism for any discrete $\hat{N}_p$-module $M$ of finite $p$-power order. Moreover, if $G$ lacks any $C_{p^\infty}$-sections, the subgroup $N$ enjoys some additional special properties with respect to its pro-$p$ topology.
\vspace{10pt}

\end{abstract}
\vspace{20pt}

\section{Introduction}

A group with {\it finite abelian section rank}-- referred to in \cite{lennox} as an {\it FAR-group}-- is a group whose abelian sections each have finite torsion-free rank as well as finite $p$-rank for every prime $p$. The present article examines circumstances under which the cohomology of a solvable FAR-
group coincides with that of either its profinite or pro-$p$ completion. With regard to profinite
completion, we establish the following theorem.

\begin{theorem} Assume $p$ is a prime and $G$ a solvable FAR-group without any subgroups isomorphic to $C_{p^\infty}$. Then the profinite completion map $G\to \hat{G}$ induces an isomorphism
\begin{displaymath} H^\ast(\hat{G},M)\to H^\ast(G,M)\end{displaymath}
for any discrete $\hat{G}$-module $M$ of finite $p$-power order. 
\end{theorem}

For a nilpotent FAR-group we prove the analogous result below concerning the cohomol-
ogy of its pro-p completion. 

\begin{theorem} Let $G$ be a nilpotent FAR-group. Assume $p$ is a prime such that $G$ fails to contain any subgroups isomorphic to $C_{p^\infty}$. Then the pro-$p$ completion map $G\to \hat{G}_p$ induces an isomorphism
\begin{equation} H^\ast(\hat{G}_p,M)\to H^\ast(G,M)\end{equation}
for any discrete $\hat{G}_p$-module $M$ of finite $p$-power order. 
\end{theorem}

If, in the above theorem, $G$ were only assumed to be solvable rather than nilpotent, then the map (1.1) might not be an isomorphism. Nevertheless, our next theorem shows that, in this case, it is still possible to find a normal subgroup of finite index with the same cohomology as its pro-$p$ completion. 

\begin{theorem}  Let $p$ be a prime. Assume $G$ is a solvable FAR-group without any subgroups isomorphic to $C_{p^\infty}$ . Then $G$ possesses a normal subgroup $N$ of finite index such that the canonical map
\begin{displaymath} H^\ast(\hat{N}_p,M)\to H^\ast(N,M)\end{displaymath}
is an isomorphism for any discrete $\hat{N}_p$-module $M$ of finite $p$-power order. 
\end{theorem}

The above theorem is proved at the beginning of Section 3 of the paper.
In the remainder of that section, we examine more closely the phenomenon exemplified by this theorem, namely, the presence in a solvable FAR-group of subgroups of finite index that are blessed with rich, orderly pro-$p$ topologies. Our final result is the following elaboration of Theorem~1.3 for the case where $G$ lacks any $C_{p^\infty}$-sections. 

\begin{theorem} Assume $G$ is a solvable FAR-group. Let $p$ be a prime such that $G$ does not have any sections isomorphic to $C_{p^\infty}$. Then $G$ contains a normal subgroup $N$ of finite index with the following three properties.
\vspace{5pt}

(i) For any section $S$ of $N$, the canonical map
\begin{displaymath} H^\ast(\hat{S}_p,M)\to H^\ast(S,M)\end{displaymath}
\indent is an isomorphism whenever $M$ is a discrete $\hat{S}_p$-module $M$ of finite $p$-power order. 
\vspace{5pt}

(ii) Each section of $N$ is an extension of a $p'$-group by a residually $p$-finite group.
\vspace{5pt}

(iii) For any $H\leq N$, the topology on $H$ induced by the pro-$p$ topology on $N$ coincides 
\indent with the full pro-$p$ topology on $H$. 
\end{theorem}

Any nilpotent FAR-group without  any $C_{p^\infty}$-sections satisfies properties (i),  (ii), and (iii). Hence Theorem 1.4 asserts the existence of a subgroup of finite index that, although not necessarily nilpotent, exhibits three salient features of a nilpotent group with regard to its pro-$p$ topology. Like Theorem 1.3, this result was inspired by the following theorem of J.~Lennox  and D. Robinson [{\bf 3}, 5.3.9]; indeed our choice of the subgroup $N$ coincides with theirs for the case of a residually finite minimax group.  

\begin{theorem} {\rm (Lennox, Robinson)} Assume $G$ is a residually finite solvable minimax group. If $p$ is a prime such that $G$ does not contain any sections isomorphic to $C_{p^\infty}$, then $G$ contains a normal subgroup $N$ of finite index that is residually $p$-finite.
\end{theorem}

\noindent For the special case of polycyclic groups, Lennox and Robinson's result was first proved by A.~\u{S}mel'kin \cite{smelkin}. 
\vspace{10pt}

\begin{acknowledgement} {\rm The key observation that led to Theorem 1.1 was that the case for $G$ torsion-free could be reduced to that of a group of rank one. This manner of argument was suggested to the author by Peter Kropholler. It is with much gratitude that the author acknowledges Peter Kropholler's invaluable contribution in providing this crucial idea, as well as in supplying other important insights.
}
\end{acknowledgement}
\vspace{10pt}

\begin{notation} {\rm A {\it section} of a group is a quotient of one of its subgroups. 

If $p$ is a prime, an abelian group is said to have {\it finite $p$-rank} if its $p$-torsion subgroup is a direct sum of finitely many cyclic and quasicyclic groups. 

If $G$ is a solvable FAR-group, then its {\it spectrum}, denoted ${\rm spec}(G)$, is the set of all primes $p$ for which $G$ has a section isomorphic to $C_{p^\infty}$. 

A solvable FAR-group is {\it reduced} if it does not contain any radicable subgroups.

If $G$ is a group, the {\it finite residual} of $G$ is the kernel of the profinite completion map $G\to \hat{G}$. Moreover, for any prime $p$, the {\it $p$-residual} of $G$ is the kernel of the pro-$p$ completion map $G\to \hat{G}_p$. 

 Let $G$ be a group and $p$ a prime. If $H$ is a subgroup of $G$ of finite index, we write $H\leq_f G$ (or $H\unlhd_f G$ in the case that $H$ is also normal). If $H$ is a subgroup of $G$ of finite $p$-power index, we write $H\leq_p G$ (or $H\unlhd_p G$ if $H$ is also normal). 

If $H$ is a subgroup of $G$ such that the topology on $H$ induced by the profinite topology on $G$ coincides with the full profinite topology on $H$, then we say that $H$ is {\it topologically embedded} in $G$ and write $H\leq_t G$. Notice that the following three statements are equivalent:

1. $H\leq_t G$.

2. For every $K\unlhd_f H$, there exists $N\unlhd_f G$ such that $N\cap H\leq K$. 

3. The map $\hat{H}\to \hat{G}$ induced by the inclusion map $H\to G$ is an injection. 

If $H$ is a subgroup of $G$ such that the topology on $H$ induced by the pro-$p$ topology on $G$ coincides with the full pro-$p$ topology on $H$, then we say that $H$ is {\it topologically p-embedded} in $G$ and write $H\leq_{t(p)} G$. As above, we observe that the following three statements are equivalent:

1. $H\leq_{t(p)} G$.

2. For every $K\unlhd_p H$, there exists $N\unlhd_p G$ such that $N\cap H\leq K$. 

3. The map $\hat{H}_p\to \hat{G}_p$ induced by the inclusion map $H\to G$ is an injection. 

If every subgroup of $G$ is topologically $p$-embedded in $G$, then we say that the pro-$p$ topology on $G$ is {\it homogeneous}, or that $G$ is {\it $p$-homogeneous}. 
}
\end{notation}
\vspace{10pt}

\section{The classes $\mathcal{G}(p)$ and $\mathcal{G}^\ast(p)$}
\vspace{10pt}

\indent This article is concerned primarily with the following class of groups.

\begin{definition} {\rm Let $p$ be a prime. Define $\mathcal{G}(p)$ to be the class of all groups $G$ such that, for each discrete $\hat{G}_p$-module $M$ of finite $p$-power order, the pro-$p$ completion map $G\to \hat{G}_p$ induces an isomorphism $H^\ast(\hat{G}_p,M)\to H^\ast(G,M)$.}
\end{definition}

Our main goal in this section is to show that, for any prime $p$, $\mathcal{G}(p)$ contains every nilpotent FAR-group without any subgroups isomorphic to $C_{p^\infty}$. In this endeavor, we will employ the following simplification of the defining property for $\mathcal{G}(p)$. 
 
 \begin{proposition} Let $p$ be a prime. A group $G$ belongs to $\mathcal{G}(p)$ if and only if, for every trivial discrete $\hat{G}_p$-module $M$ of finite $p$-power order, the canonical map $H^\ast(\hat{G}_p,M)\to H^\ast(G,M)$ is an isomorphism.
\end{proposition}

\begin{proof} If $M$ is a finite discrete $\hat{G}_p$-module and $N=C_G(M)$, then there is an exact sequence $1\longrightarrow N\longrightarrow G\longrightarrow Q\longrightarrow~1$, where $Q$ is a finite $p$-group. This yields an exact sequence $1\longrightarrow \hat{N}_p\longrightarrow \hat{G}_p\longrightarrow Q\longrightarrow~1$ of pro-$p$ groups. Hence the proposition follows immediately from the Lyndon-Hochschild-Serre (LHS) spectral sequences for the two extensions. 
\end{proof}

Closely related to the class $\mathcal{G}(p)$ is the class defined below.  

\begin{definition} {\rm Let $p$ be a prime.  Define $\mathcal{G}^\ast(p)$ to be the class of all groups $G$ such that, for each discrete $\hat{G}$-module $M$ of finite $p$-power order, the profinite completion map $G\to \hat{G}$ induces an isomorphism $H^\ast(\hat{G},M)\to H^\ast(G,M)$.}
\end{definition}

\noindent As a consequence of our theorem on nilpotent groups in $\mathcal{G}(p)$, we will deduce that  $\mathcal{G}^\ast(p)$ contains every solvable FAR-group that does not have any subgroups isomorphic to $C_{p^\infty}$. 

In [{\bf 6}, Exercises 1 \&\  2, pp. 15-16], J-P. Serre provides an enlightening introduction to groups that belong to $\mathcal{G}^\ast(p)$ for all primes $p$. He refers to such groups as ``good groups," the appellation which inspired our designations $\mathcal{G}(p)$ and $\mathcal{G}^\ast(p)$. Using Serre's insights, T.~Weigel [{\bf 8}, Proposition 3.1] establishes the following connection between the classes $\mathcal{G}(p)$ and $\mathcal{G}^\ast(p)$.

\begin{lemma}{\rm (Weigel)} Let $G$ be a group and $p$ a prime. If every subgroup of finite index in $G$ belongs to $\mathcal{G}(p)$, then $G$ is contained in $\mathcal{G}^\ast(p)$. 
\end{lemma}

In addition to $\mathcal{G}(p)$ and $\mathcal{G}^\ast(p)$ we will refer to two other classes of groups.

\begin{definition}{\rm Let $\mathcal{FS}$ be the class of all groups $G$ such that, for each $n\in \mathbb N$, $G$ has only finitely many subgroups of index $n$.} 
\end{definition}

\begin{definition}{\rm Let $\mathcal{FH}$ be the class of all groups $G$ such that, for any finite $G$-module $M$, $H^n(G,M)$ is finite for all $n\geq 0$.}
\end{definition}

A group $G$ is in the class $\mathcal{FS}$ if and only if the set ${\rm Hom}(G,F)$ is finite for every finite group $F$. From this observation it is easy to see that $\mathcal{FS}$ is closed under group extensions.
Moreover, every solvable FAR-group has a normal series of finite length whose factor groups are either infinite cyclic or torsion abelian FAR-groups. Therefore, since the latter two sorts of groups are both in $\mathcal{FS}$, we can quickly deduce the following proposition. 

\begin{proposition} Every solvable FAR-group belongs to $\mathcal{FS}$.
\end{proposition}

From the LHS spectral sequence it is apparent that $\mathcal{FH}$, too, is closed under group extensions. As a result, using the same type of normal series as above, we can make the following assertion.

\begin{proposition}{\rm (Robinson [{\bf 5}, Lemma 4.4])} Every solvable FAR-group belongs to $\mathcal{FH}$.
\end{proposition}

\begin{proof} The presence of a normal series of finite length whose factor groups are either infinite cyclic or torsion abelian FAR-groups permits a reduction to the case of a torsion abelian FAR-group. Let $G$ be such a group and $M$ a finite $G$-module. Without incurring any significant loss of generality, we may make three further reductions: (1) the action of $G$ on $M$ is trivial; (2) $M$ has $p$-power order for some prime $p$; (3) $G\cong C_{p^\infty}$. Since integral homology commutes with direct limits, $H_n(G)$ is a $p$-group of rank one for $n\geq 1$, making it either a finite cyclic $p$-group or a quasicyclic one. In either case, it follows from the universal coefficient theorem that $H^n(G,M)$ is finite for $n\geq 1$. 
\end{proof}

\begin{remark}{\rm Robinson's  proof of the above proposition differs slightly from ours and reveals more about the cohomology of quasicyclic groups.}
\end{remark}

The relevance of  $\mathcal{FH}$ to the study of the class $\mathcal{G}^\ast (p)$ is made clear by the following lemma, gleaned from  [{\bf 6}, Exercise 2, p. 16]. 

\begin{lemma}{\rm (Serre)}  Let $1\longrightarrow N\longrightarrow G\longrightarrow Q\longrightarrow 1$ be a group extension in which $N\leq_t G$. Suppose, further, that $N$ belongs to $\mathcal{FH}$. If both $N$and $Q$ are in $\mathcal{G}^\ast(p)$, then $G$ is contained in $\mathcal{G}^\ast(p)$. 
\end{lemma}

\begin{proof} A comparison of the LHS spectral sequences for $1\longrightarrow N\longrightarrow G\longrightarrow Q\longrightarrow 1$ and $1\longrightarrow \hat{N}\longrightarrow \hat{G}\longrightarrow \hat{Q}\longrightarrow 1$ yields the conclusion.
\end{proof}

When applying the above lemma it is essential that the normal subgroup be topologically embedded in the group. One situation where this holds in a solvable FAR-group is described below. 

\begin{lemma} Let $G$ be a solvable FAR-group and $N\unlhd G$. If $N$ is closed in the profinite topology on $G$, then $N\leq_t G$. 
\end{lemma}

\begin{proof} Let $K\unlhd_f N$. Take $K'$ to be the intersection of all the conjugates of $K$ in $G$. Then $K'\unlhd G$; moreover, since $N$ belongs to $\mathcal{FS}$, $[N:K']$ is finite. As a residually finite group, $G/N$ is reduced. Therefore, $G/K'$, too, is reduced, so that, by [{\bf 3}, Corollary 5.3.2], it is residually finite. Thus there exists $L\unlhd_f G$ such that $K'\leq L$ and $L\cap N\leq K'$. It follows, then, that $N\leq_t G$. 
\end{proof}

Lemma 2.5 has the following pro-$p$ analogue, proved in much the same way. 

\begin{lemma} Assume $p$ is a prime. Let
$1\longrightarrow N\longrightarrow G\longrightarrow Q\longrightarrow 1$
be a group extension in which $N\leq_{t(p)} G$. Suppose, further, that $N$ belongs to $\mathcal{FH}$. If $N$and $Q$ are in $\mathcal{G}(p)$, then $G$ belongs to $\mathcal{G}(p)$. 
\end{lemma}

Lemma 2.7 requires that the normal subgroup in question be topologically $p$-embedded in the ambient group. One pair of conditions that implies this property is described below-- similar observations may be found in  \cite{schick}.

\begin{lemma}  Assume $p$ is a prime. Let $G$ be a group and $A\leq Z(G)$. If $G/A$ belongs to $\mathcal{G}(p)$, then $A\leq_{t(p)} G$. 
 \end{lemma}

\begin{proof} To begin with, we consider the case where $A$ is a finite $p$-group.  Set $Q=G/A$ and let  $\xi\in H^2(Q,A)$ be the cohomology class of the extension $1\longrightarrow A\longrightarrow G\longrightarrow Q\longrightarrow 1$. The group $H^2(\hat{Q}_p,A)$ is the direct limit of the cohomologies with coefficients in $A$ of all the $p$-finite quotients of $Q$. In view of the surjectivity of the map $H^2(\hat{Q}_p,A)\to H^2(Q,A)$, this means that there exists $R\unlhd_p Q$ such that the image of $\xi$ in $H^2(R,A)$ is trivial. Thus we have $U\unlhd_p G$ containing $A$ such that $U$ splits over $A$. Let $V\leq  U$ such that $U=AV$ and $A\cap V=1$. Then $V\unlhd_p U$. Take $N$ to be the intersection of all the conjugates of $V$ in $G$. Then $N\unlhd G$ and $N\cap A=1$.  Notice that each conjugate of $V$ in $G$ is a normal subgroup of $U$ with finite $p$-power index; also, there are only finitely many such conjugates. These two observations, together, imply that $[U:N]$ is a finite $p$-power, so that $N\unlhd_p G$. It follows, then, that $A\leq_{t(p)} G$. Now we treat the case where $A$ is an arbitrary abelian group.  Let $B\leq_p A$.  By the first case, $A/B\leq_{t(p)} G/B$. Hence there exists $M\unlhd_p G$ with $M\cap A\leq B$.  Consequently, $A\leq_{t(p)} G$. 
\end{proof}

Lemmas 2.7 and 2.8 allow us to conclude that the class $\mathcal{G}(p)$ is closed under the formation of certain central extensions.

\begin{proposition}  Assume $p$ is a prime. Let $1\longrightarrow A\longrightarrow G\longrightarrow Q\longrightarrow 1$ be a central group extension such that $A$ belongs to $\mathcal{FH}$. If $A$ and $Q$ are both contained in $\mathcal{G}(p)$, then $G$ also belongs to $\mathcal{G}(p)$. 
 \end{proposition}
 
With the above closure property, we can establish that every torsion-free nilpotent FAR-group is in $\mathcal{G}(p)$. 
 
  \begin{lemma} If $G$ is a torsion-free nilpotent  FAR-group, then $G$ belongs to the class $\mathcal{G}(p)$ for each prime $p$. 
\end{lemma}

\begin{proof} The group $G$ has a central series whose factors are each torsion-free abelian groups of rank one. Thus, in view of Proposition 2.9, it suffices to consider the case where $G$ is torsion-free abelian of rank one. Let $M$ be a 
trivial discrete $\hat{G}_p$-module of finite $p$-power order. The cohomologies of any group and its pro-$p$ completion with trivial $p$-finite coefficients are always isomorphic in dimensions zero and one. Hence we only need to look at cohomology in dimension $n\geq 2$. Because $\hat{G}_p$ is either trivial or isomorphic to $\hat{\mathbb Z}_p$, $H^n(\hat{G}_p,M)=0$. Applying the universal coefficient theorem, we obtain an exact sequence
\begin{displaymath}
 \begin{CD}
0 @>>> {\rm Ext}(H_{n-1}(G), M) @>>> H^n(G,M) @>>> {\rm Hom}(H_n(G), M) @>>>
0.\end{CD} \end{displaymath}
\noindent Since $G$ is a direct limit of infinite cyclic groups and homology commutes with direct limits, $H_{n-1}(G)$ is torsion-free. Hence  
${\rm Ext}(H_{n-1}(G), M)=0$. In addition, we have $H_n(G)=0$, so that $H^n(G,M)=0$. Therefore, $H^\ast(\hat{G}_p,M)\cong H^\ast(G,M)$. 
\end{proof}

Before extending Lemma 2.10 to nilpotent FAR-groups without any $C_{p^\infty}$-subgroups, we need to investigate FAR-groups that are both torsion and abelian. 

\begin{lemma} Let $p$ be a prime. Assume $A$ is a torsion abelian group of finite $p$-rank.
If $A$ fails to contain any subgroups isomorphic to $C_{p^\infty}$, then $A$ belongs to the class $\mathcal{G}(p)$. 
\end{lemma}

\begin{proof} Let $A_p$ and $A_{p'}$ be the $p$-torsion and ${p'}$-torsion subgroups, respectively, of $A$. Assume $M$ is a discrete trivial $\hat{A}_p$-module of finite $p$-power order. Then $H^\ast(A,M) \cong H^\ast(A_p,M)$. Therefore, since $\hat{A}_p=A_p$, we have $H^\ast(\hat{A}_p,M)\cong H^\ast(A,M)$. 
\end{proof}

Now we have everything in place to prove the first of our two theorems. 

\begin{theorem} Let $p$ be a prime. If $G$ is a nilpotent FAR-group without any subgroups isomorphic to $C_{p^\infty}$, then $G$ belongs to $\mathcal{G}(p)$.
\end{theorem}

\begin{proof} We argue by induction on the nilpotency class of $G$. First we assume that $G$ is abelian. Let $T$ be the torsion subgroup of $G$. Then $T$ belongs to $\mathcal{G}(p)$ by Lemma 2.11. Also, $G/T$ is in $\mathcal{G}(p)$ by Lemma 2.10. Consequently, we can conclude from Proposition 2.9 that $G$ belongs to $\mathcal{G}(p)$. Next we consider the case where the nilpotency class exceeds one.  Initially, we suppose that $G$ is residually finite. In this case, $Z(G)$ is a closed subgroup of $G$ with respect to the profinite topology, implying that $G/Z(G)$ is residually finite. Therefore, by the inductive hypothesis, $G/Z(G)$ is in $\mathcal{G}(p)$. In addition, by the base case, $Z(G)$ belongs to $\mathcal{G}(p)$. Proposition 2.9, then, yields that $G$ is contained in $\mathcal{G}(p)$. For the general case, we let $R$ be the finite residual of $G$. According to [{\bf 3}, 5.3.1], $R$ is radicable. Moreover, by the residually finite case, $G/R$ is contained in $\mathcal{G}(p)$. Let $T$ be the torsion subgroup of $R$. In a group that is both radicable and nilpotent, all elements of finite order are contained in the center (see [{\bf 3}, p. 98]).  Hence $T\leq Z(R)$. Therefore, by Lemma 2.11, $T$ belongs to $\mathcal{G}(p)$. Also, $R/T$ is in $\mathcal{G}(p)$ by virtue of Lemma 2.10.  Invoking Proposition 2.9 once again, we obtain that $R$ belongs to $\mathcal{G}(p)$. Furthermore, as a radicable subgroup, $R$ must be  topologically $p$-embedded in $G$. It follows, then, from Lemma 2.7 that $G$ is in $\mathcal{G}(p)$.  
\end{proof}

From Theorem 2.12 and Lemma 2.2 we obtain

\begin{corollary}  Let $p$ be a prime. If $G$ is a nilpotent FAR-group without any subgroups isomorphic to $C_{p^\infty}$, then $G$ belongs to
$\mathcal{G}^\ast(p)$. 
\end{corollary}

Corollary 2.13 allows us to prove our second theorem.

\begin{theorem} Assume $p$ is a prime. If $G$ is a solvable FAR-group without any subgroups isomorphic to $C_{p^\infty}$, then $G$ is contained in $\mathcal{G}^\ast(p)$.
\end{theorem}

\begin{proof} The proof is by induction on the derived length. With the result already established for $G$ abelian, we proceed directly to the case where the derived length of $G$ is greater than one. 
First we suppose that $G$ is residually finite. Let $A$ be the closure in the profinite topology of the last nontrivial group in the derived series of $G$. Then $A$ is abelian and, by virtue of Lemma 2.6,  $A\unlhd_t G$. By the base case, $A$ is in $\mathcal{G}^\ast(p)$; moreover, by the inductive hypothesis, $G/A$ belongs to $\mathcal{G}^\ast(p)$. Lemma 2.5, then, yields that  $G$ is in $\mathcal{G}^\ast(p)$. For the general case, we let $R$ be the finite residual of $G$. Then $G/R$ belongs to $\mathcal{G}^\ast(p)$ by the residually finite case. According to [{\bf 3}, 5.3.1], $R$ is both nilpotent and radicable. In view of Corollary 2.13, the former property implies that $R$ is in $\mathcal{G}^\ast(p)$.  In addition, $R$'s radicability guarantees that $R\leq_t G$. Hence $G$ is in $\mathcal{G}^\ast (p)$ by Lemma 2.5. 
\end{proof}

\begin{remark} {\rm The case of Theorem 2.14 where $G$ is virtually torsion-free was first proved by P. Kropholler  \cite{kropholler}.}
\end{remark}
\vspace{10pt}

We conclude this section with an example illustrating the necessity of the finite-rank hypothesis in both Theorems 2.12 and 2.14.

\begin{example} {\rm Let $G$ be the restricted direct product  of countably many copies of $C_{\infty}$. We will show that $G$ is not contained in either $\mathcal{G}(p)$ or $\mathcal{G}^\ast(p)$ for any prime $p$. To accomplish this, we employ the group $\Gamma$ with the following presentation:

\begin{displaymath} {\Gamma  =  \langle a, x_1, x_2, \cdots  \ |\  a^p=1; \ \ \ \ [a,x_i]=1, [x_i,x_{i+1}]=a, \ \ \mbox{\rm and} \ \ [x_i,x_j]=1 \ \ \mbox{\rm for all $i, j \in \mathbb N$ with $|i-j|\neq 1$} \rangle.}\end{displaymath}
This group is nilpotent of class 2 and fits into a central extension $1\to \langle a\rangle \to \Gamma\to G\to 1$. If $G$ were in $\mathcal{G}(p)$, then, according to Lemma 2.8,  $\Gamma$ would necessarily contain a normal subgroup of $p$-power index that intersected $\langle a\rangle$ trivially. We claim that no such subgroup exists, implying that $G$ must lie outside $\mathcal{G}(p)$. To establish this, let $N\unlhd_p \Gamma$. 
Then $x_ix_j^{-1}\in N$ for some pair of distinct natural numbers $i$ and $j$. If $j\neq i+2$, then
$$a = [x_ix_j^{-1},x_{i+1}]\in N.$$ 
Further, if $j=i+2$, then
$$a = [x_{i+3},x_ix_j^{-1}]\in N.$$
Thus our claim follows. A similar argument verifies that $G$ is also not a member of $\mathcal{G}^\ast(p)$. Moreover, our reasoning can be easily adapted to establish that the restricted direct product of countably many copies of $C_p$ is another example of an abelian group outside both $\mathcal{G}(p)$ and $\mathcal{G}^\ast(p)$.
}
\end{example}

\section{Subgroups of finite index with rich pro-$p$ topologies}
\vspace{10pt}

Our aim in this section is to find subgroups of finite index in a solvable FAR-group that have pro-$p$ topologies replicating the properties of the pro-$p$ topology of a nilpotent group. In our first theorem we establish that any solvable FAR-group without $C_{p^\infty}$-subgroups has a subgroup of finite index in $\mathcal{G}(p)$. 

\begin{theorem}  Let $p$ be a prime. Assume $G$ is a solvable FAR-group without any subgroups isomorphic to $C_{p^\infty}$. Then there exists $N\unlhd_f G$ such that every subgroup of finite index in $N$ belongs to the class $\mathcal{G}(p)$.  
\end{theorem}

\begin{proof} First we treat the case where $G$ is residually finite.  
We proceed by induction on the derived length of $G$, the abelian case having already been established in Theorem 2.12. For the inductive step, let $A$ be the closure in the profinite topology of the last nontrivial group in the derived series. Then $A$ is an abelian characteristic subgroup of $G$, and $G/A$ is residually finite. By the inductive hypothesis, there exists $M\unlhd_f G$  such that $A\leq M$ and every subgroup of $M/A$ with finite index belongs to $\mathcal{G}(p)$. Now we set $N=C_M(A/A^p)$. Then $N\unlhd_f G$. We will prove that, in addition, $A\leq_{t(p)} N$. To begin with, we show by induction on $i$ that $N/A^{p^i}$ is in $\mathcal{G}(p)$ for each $i\geq 0$. The case for $i=0$ follows from our selection of $M$. If $i\geq 1$, we invoke the central extension
$$1\longrightarrow A^{p^{i-1}}/A^{p^i}\longrightarrow N/A^{p^i}\longrightarrow N/A^{p^{i-1}}\longrightarrow 1,$$
deducing on the basis of Proposition 2.9 that $N/A^{p^i}$ is in $\mathcal{G}(p)$. Since $N/A^{p^i}$ belongs to $\mathcal{G}(p)$ for all $i\geq 0$, it follows from Lemma 2.8 that $A^{p^{i-1}}/A^{p^i}$ is topologically $p$-embedded in $N/A^{p^i}$ for all $i\geq 1$. Hence, for each $i\geq 1$, there exists $U_i\unlhd_p N$ such that $A^{p^i}\leq U_i$ and $U_i\cap A^{p^{i-1}}\leq A^{p^i}$. For each $k\in \mathbb N$, let $V_k$ be the intersection of $U_1,\cdots , U_k$. Then $V_k\unlhd_p N$ and $V_k\cap A\leq A^{p^k}$. Therefore, $A\leq_{t(p)} N$. Now let $H\leq_f N$. Any subgroup of finite index in an abelian group is topologically $p$-embedded; hence  $H\cap A\leq_{t(p)} A$. As a result,  $H\cap A\leq_{t(p)} H$. Moreover, $H/H\cap A\cong HA/A\leq_f M/A$, so that $H/H\cap A$ is in $\mathcal{G}(p)$. Therefore, by Lemma 2.7, $H$ is contained in $\mathcal{G}(p)$. 

For the general case, we let $R$ be the finite residual of $G$. By [{\bf 3}, 5.3.1], $R$ is both nilpotent and radicable. From the first property we can conclude that $R$ belongs to $\mathcal{G}(p)$. Moreover, $R$'s radicability ensures that it is topologically $p$-embedded in $G$.  By the residually finite case, $G$ contains a normal subgroup $N$ of finite index such that $R\leq N$ 
and all the subgroups of $N/R$ with finite index are in the class $\mathcal{G}(p)$. Suppose $H\leq_f N$. Then $R\leq_{t(p)} H$, and $H/R$ is in $\mathcal{G}(p)$. Therefore, by Lemma 2.7, $H$ is in $\mathcal{G}(p)$.
\end{proof}

In the remainder of this section, we discern the presence of subgroups of finite index that share even more properties with nilpotent groups. Our goal is to prove the following theorem.

\begin{theorem} Assume $G$ is a solvable FAR-group. For any prime $p\notin {\rm spec}(G)$, there exists $N\unlhd_f G$ with the following three properties. 
\vspace{5pt}

(i) Each section of $N$ belongs to the class $\mathcal{G}(p)$. 
\vspace{5pt}

(ii) The $p$-residual of any section of $N$ is a $p'$-group.
\vspace{5pt}

(iii) The pro-$p$ topology on $N$ is homogeneous. 
\end{theorem}

Before embarking on the proof of Theorem 3.2, we require some preliminary lemmas concerning the pro-$p$ topology of a group. The first of these is very simple to deduce; hence we omit its proof. 

\begin{lemma} Assume $p$ is a prime, $G$ a group, and $S$ a section of $G$. If $G$ is $p$-homogeneous, then so is $S$.
\end{lemma}

\begin{lemma} Assume $G$ is a group in the class $\mathcal{FS}$. Let $p$ be a prime and $N\unlhd_p G$. If $N$ is $p$-homogeneous, then so is~$G$.
\end{lemma}

\begin{proof} Let $H\leq G$ and $K\unlhd_p H$. Then $$K\cap N\unlhd_p H\cap N\leq N.$$
Since $N$ is $p$-homogeneous, we have a subgroup $L\unlhd_p N$ such that $L\cap H\leq K$.
Let $L'$ be the intersection of all the conjugates of $L$ in $G$. Then $L'\unlhd G$ and $L'\cap H\leq K$.
Moreover, each conjugate of $L$ in $G$ is normal in $N$ with finite $p$-power index, and, since $G$ is in $\mathcal{FS}$, there are only finitely many such conjugates. It follows, then, that $L'\unlhd_p N$, implying that $L'\unlhd_p G$. Therefore, $H\leq_{t(p)} G$. 
\end{proof}

\begin{lemma} Let $p$ be a prime. Assume $1\longrightarrow N\longrightarrow G\longrightarrow Q\longrightarrow 1$ is a short exact sequence of groups such that  $N\leq_{t(p)} G$ and $G$ belongs to $\mathcal{FS}$. If both $N$ and $Q$ are $p$-homogeneous, then $G$ is as well. 
\end{lemma}

\begin{proof} First assume that $N$ is a finite $p$-group. Then there exists $U\unlhd_p G$ such that 
$U\cap N=1$. Since $U$ is isomorphic to a subgroup of $Q$, it is $p$-homogeneous. Thus, by Lemma 3.4, $G$ is $p$-homogeneous.
  
Now we consider the general case. Let $H\leq G$ and $K\unlhd_p H$. Then $K\cap N\unlhd_p H\cap N$. Since $N$ is $p$-homogeneous and $N\leq_{t(p)} G$, there exists $L\unlhd_p G$ such that 
$L\cap H\cap N\leq K\cap N$. Let $M=N\cap L$. By the case proved above, $G/M$ is $p$-homogeneous. Thus there is a subgroup $P\unlhd_p G$ such that $M\leq P$ and $P\cap HM\leq KM$. Suppose $x\in P\cap H$. Then $x=km$ for $k\in K$ and $m\in M$. It follows that $m\in H\cap M$, which implies that $m\in K$. Consequently, $P\cap H\leq K$. Therefore, $H\leq_{t(p)} G$. 
\end{proof}

\begin{lemma} Let $p$ be a prime. Assume $1\longrightarrow N\longrightarrow G\longrightarrow Q\longrightarrow 1$ is a short exact sequence of groups such that  $N\leq_{t(p)} G$. If the $p$-residuals of both $N$ and $Q$ are $p'$-groups, then the $p$-residual of $G$ is also a $p'$-group. 
\end{lemma}

\begin{proof}  Consider the commutative diagram 
\begin{displaymath} \begin{CD}
1 @>>> N@>>> G @>>> Q @>>>1\\
&& @VVV @VVV @VVV && \\
1 @>>> \hat{N}_p @>>> \hat{G}_p @>>> \hat{Q}_p @>>>1
\end{CD} \end{displaymath}
with exact rows whose vertical maps are $p$-completion maps.  Since the kernels of the first and third vertical maps are both $p'$-groups, this holds for the second as well. 
\end{proof}

Next we observe that any nilpotent FAR-group without $p$ in its spectrum fulfills properties (i), (ii), and (iii) of Theorem 3.2. 

\begin{proposition} Let $G$ be a nilpotent FAR-group. If $p$ is a prime such that $p\notin {\rm spec}(G)$, then $G$ satisfies the following three properties.
\vspace{5pt}

(i) Every section of $G$ is in $\mathcal{G}(p)$.
\vspace{5pt}

(ii) The $p$-residual of $G$ is a $p'$-group.
\vspace{5pt}

(iii) The pro-$p$ topology on $G$ is homogeneous. 
\end{proposition}

\begin{proof} Property (i) follows at once from Theorem 2.12. To establish the other two properties, we observe that the group $G$ has a central series whose factor groups are each either a copy of $C_{\infty}$ or a torsion abelian FAR-group without $p$ in its spectrum. Moreover, the latter two species of groups satisfy both (ii) and (iii). Therefore, by Lemma 2.8, Lemma 3.5, and Lemma 3.6, $G$ satisfies (ii) and (iii). 
\end{proof}

Now we are prepared to prove Theorem 3.2.

\begin{proof}[Proof of Theorem 3.2]  The argument proceeds very much like the proof of Theorem 3.1. As before, we reason by induction on the derived length of $G$. Having already disposed of the abelian case in Proposition 3.7, we assume that the derived length of $G$ exceeds one. Let $A$ be the last term in the derived series of $G$. By the inductive hypothesis, $G$ contains a normal subgroup $M$ of finite index such that  $A\leq M$ and $M/A$ satisfies properties (i), (ii), and (iii). Again we take $N=C_M(A/A^p)$. As argued in the proof of Theorem 3.1, we have that $A\leq_{t(p)} N$. Therefore, by Lemma 3.5, the pro-$p$ topology on $N$ is homogeneous. To verify the other two properties, let $\bar{N}$ be a section of $N$. Then $\bar{N}$ fits into an extension of the form 
$1\longrightarrow \bar{A}\longrightarrow \bar{N}\longrightarrow  \overline{N/A}\longrightarrow 1$, where $\bar{A}$ and $\overline{N/A}$ are sections of $A$ and $N/A$, respectively. Also, the $p$-homogeneity of $N$ ensures that $\bar{A}\leq_{t(p)} \bar{N}$. Thus $\bar{N}$ must be in $\mathcal{G}(p)$ by virtue of Lemma 2.7, and the $p$-residual of $\bar{N}$ is a $p'$-group in view of Lemma 3.6. 
\end{proof}

\begin{remark}{\rm We point out that a special case of Theorem 3.2 appears as [{\bf 1}, Corollary 2.4]. There it is observed that a polycyclic group contains a normal subgroup of finite index satisfying property (ii).}
\end{remark}


\begin{thebibliography}{8}


\bibitem[{\bf 1}]{smith}{\sc G. Cutolo} and {\sc H. Smith}. A note on polycyclic residually finite-$p$ groups. {\it Glasg. Math. J.} {\bf 52} (2010), no. 1, 137-143.

\bibitem[{\bf 2}]{kropholler}{\sc P. Kropholler.} Private communication.

\bibitem[{\bf 3}]{lennox}{\sc J. Lennox} and {\sc D. Robinson}. {\it The Theory of Infinite Soluble Groups} (Oxford, 2004).

\bibitem[{\bf 4}]{schick}{\sc P. Linnell} and {\sc T. Schick}. Finite group extensions and the Atiyah conjecture. {\it J. Amer. Math. Soc.} {\bf 20} (2007), 1003-1051.

\bibitem[{\bf 5}]{robinson}{\sc D. Robinson.} On the cohomology of soluble groups of finite rank. {\it J. Pure Appl. Algebra}~{\bf 6} (1975), 155-64. 

\bibitem[{\bf 6}]{serre}{\sc J-P. Serre}. {\it Galois Cohomology} (Springer, 1997).

\bibitem[{\bf 7}]{smelkin}{\sc A. \u{S}mel'kin}. Polycyclic groups. {\it Sibirsk Mat. Z.} {\bf 9} (1968), 234-235.

\bibitem[{\bf 8}]{weigel}{\sc T. Weigel}. On profinite groups with finite abelianizations. {\it Sel. Math., New Ser.} {\bf 13} (2007), 175-181. 

\end{thebibliography}
\end{document}